\documentclass[final,1p,times]{elsarticle}

\usepackage{graphicx}
\usepackage{epstopdf}

\usepackage{amssymb}
\usepackage{amsthm}
\usepackage{amsmath}
\usepackage{url}
\usepackage{hyperref}
\usepackage{multirow}
\usepackage{epstopdf}
\usepackage{booktabs}
\usepackage{mathrsfs}

\journal{Elsevier}

\begin{document}

\begin{frontmatter}



\title{A Galerkin FE method for elliptic optimal control problem governed by 2D space-fractional PDEs}

\cortext[cor1]{Corresponding author}
\author[focal1]{X. G. Zhu\corref{cor1}}
\ead{zhuxg590@yeah.net}

\address[focal1]{School of Science, Shaoyang University, Shaoyang, Hunan 422000, P.R. China}

\begin{abstract}
In this paper, we propose a Galerkin finite element method for the elliptic optimal
control problem governed by the Riesz space-fractional PDEs on 2D domains with control variable being
discretized by variational discretization technique. The optimality condition is derived
and priori error estimates of control, costate and state variables are successfully established.
Numerical test is carried out to illustrate the accuracy performance of this approach.
\end{abstract}


\end{frontmatter}


\newtheorem{theorem}{Theorem}[section]
\newtheorem{lemma}{Lemma}[section]
\newtheorem{definition}{Defition}[section]
\newtheorem{remark}{Remarks}[section]
\newtheorem{assu}{Assumption}[section]
\renewcommand{\theequation}{\arabic{section}.\arabic{equation}}

\section{Introduction}\label{s1}

 The optimal control problems (OCPs) governed by fractional partial differential equations (PDEs) forms a new branch
 in the area of optimal control, which recently have gained explosive interest and enjoy great potential
 in the applications as diverse as temperature control, environmental engineering, crystal growth,
 disease transmission and so forth \cite{ref01,ref02}.


In this study, we consider the distributed quadratic fractional OCPs:
\begin{align}\label{eq01}
  \min\limits_{q\in \mathscr{K}}J(u,q):=\frac{1}{2}||u(x,y)-u_d(x,y)||^2_{L^2(\Omega)}
             + \frac{\gamma}{2}||q(x,y)||^2_{L^2(\Omega)},
\end{align}
subjected to the 2D elliptic Riesz fractional PDEs
\begin{align}\label{eq02}
\left\{
             \begin{array}{lll}\displaystyle
  \kappa_1\frac{\partial^{\alpha} u(x,y)}{\partial|x|^{\alpha}}
  +\kappa_2\frac{\partial^{\alpha} u(x,y)}{\partial|y|^{\alpha}}=g(x,y)+q(x,y), \quad (x,y) \in\Omega, \\
       u(x,y)=0, \quad (x,y)\in\partial\Omega,
             \end{array}
        \right.
\end{align}
where $\Omega=(a,b)\times(c,d)$, $\kappa_1,\ \kappa_2,\ \gamma\in \mathbb{R}^+$, $1<\alpha< 2$,
$\mathscr{K}$ is a closed convex set and $u_d(x,y)$ is the desired state.
The fractional derivatives have the weakly singular convolution form:
\begin{align*}
&\frac{\partial^{\alpha} u(x,y)}{\partial|x|^{\alpha}}=\frac{-1}{2\cos(\frac{\pi\alpha}{2})}\Big\{
        {_{a}^{L}D^{\alpha}_x} u(x,y)+^{R}_{x}D^{\alpha}_b u(x,y)\Big\},\\
&^{L}_{a}D^{\alpha}_x u(x,y) = \frac{1}{\Gamma(2-{\alpha})}\frac{d^2}{d x^2}
    \int^x_a(x-\omega)^{1-{\alpha}}u(\omega,y)d\omega,\\
&^{R}_{x}D^{\alpha}_b u(x,y) = \frac{1}{\Gamma(2-{\alpha})}\frac{d^2}{d x^2}
    \int^b_x(\omega-x)^{1-{\alpha}}u(\omega,y)d\omega,
\end{align*}
and so is $\frac{\partial^{\alpha} u(x,y)}{\partial|y|^{\alpha}}$ with regard to $y$.

In the past decades, the OCPs governed by PDEs have been widely investigated and a large collection
of works on their numerical algorithms  have been done, which cover
spectral method \cite{ref20}, FE method \cite{ref12,ref13,ref14,ref15}, mixed FE method \cite{ref16,ref17},
least square method \cite{ref11}, variational discretization method \cite{ref19,ref18} and some other niche methods.
However, the discussions on fractional OCPs have been rarely reported. The difficulty consisting in finding
their numerical solutions not only lies in the nonsmoothness caused by the inequality constraints on control or state, but
also the vectorial convolution in fractional derivatives, which bring enormous challenge in the endeavor of
numerical schemes and theoretical analysis. Hence, it is of great significance to study the numerical
methods for fractional OCPs. In \cite{ref26}, Mophou studied the first-order optimality condition
for the OCPs governed by time-fractional diffusion equations. In \cite{ref04}, Ye and Xu derived the optimality condition
for the time-fractional OCPs with state integral constraint and developed a spectral method.
Zhou and Gong proposed a fully discrete FE scheme to solve the time-fractional OCPs \cite{ref05}.
Du et al. combined the finite difference method and gradient projection algorithm to obtain a fast scheme
for the OCPs governed by space-fractional PDEs \cite{ref06}. Zhou and Tan addressed a fully discrete FE scheme for the
space-fractional OCPs \cite{ref09}. Zhang et al. proposed the space-time discontinuous Galerkin FE methods for
the time-fractional OCPs \cite{ref07,ref08}. Gunzburger and Wang proposed a fully discrete FE scheme along
with convolution quadrature for the time-fractional OCPs \cite{ref10}.
However, these works are limited to 1D or time-fractional OCPs. Due to the difficulty in constructing algorithm
and theoretical analysis, there is no study reported on multi-dimensional space-fractional OCPs. Inspired by this,
we propose a Galerkin FE scheme for the elliptic OCPs governed by 2D space-fractional PDEs, where the control variable
is discretized by variational discretization technique because the inequality constraints always lead to low regularity.
The first-order optimality condition is  derived and the priori error estimates
for the control, costate and state variables are rigorously analyzed.

The rest of this paper are organized as follows. In Section \ref{s2}, we derive the first-order
optimality condition for Eqs. (\ref{eq01})-(\ref{eq02}) and in Secton \ref{s3}, we propose a fully discrete FE scheme
for the optimality system. In Section \ref{s4}, we establish the priori error estimates for the control, costate and state
and finally, numerical tests are included to confirm our results.

\section{Optimality condition} \label{s2}

To begin with, we define $H^{\mu}_0(\Omega)$ by the closure in $C^\infty_0(\Omega)$
with respect to the fractional Sobolev norm $||\cdot||_{H^\mu(\Omega)}$ defined by
$||u||_{H^\mu(\Omega)}=\big( ||u||^2_{L^2(\Omega)}+|u|^2_{H^\mu(\Omega)} \big)^{1/2}$,
$|u|_{H^\mu(\Omega)}=|| \ |\boldsymbol{\omega}|^\mu \mathcal{F}[\tilde{u}]||_{L^2(\Omega)}$
with $1<\mu<2$ and $\mathcal{F}[\tilde{u}]$ being the Fourier transform of zero extension of $u$ outside $\Omega$.

Consider the model of the 2D space-fractional OCPs:
\begin{align}\label{eq031}
  \textrm{Minimize}\ J(u,q) \ \textrm{subjected to Eq.}\ (\ref{eq02}),\quad (q,u)\in \mathscr{K}\times L^2(\Omega),
\end{align}
with the pointwise constraints on control variable, i.e., 
\begin{equation*}
  \mathscr{K}=\{q\in L^2(\Omega):\ v_1\leq q(x,y)\leq v_2 \ \textrm{a.e. in}\ \Omega,\ v_1,\ v_2 \in \mathbb{R}\}.
\end{equation*}

\begin{lemma}\cite{ref21} \label{lemma01}
If $1<\mu<2$, $u,\ \chi\in H_0^{\mu}(\Omega)$, then we have
\begin{align*}
( {^{L}_a}D_x^{\mu}u,\chi)=( {^{L}_a}D_x^{\frac{\mu}{2}}u,{^{R}_x}D_b^{\frac{\mu}{2}}\chi),
\quad ( {^{R}_x}D_b^{\mu}u,\chi)=( {^{R}_x}D_b^{\frac{\mu}{2}}u,{^{L}_a}D_x^{\frac{\mu}{2}}\chi),
\end{align*}
and the similar results exist for the fractional derivatives in $y$-direction.
\end{lemma}

\begin{theorem}\label{th01}
The fractional OCPs (\ref{eq031}) have a unique pair $(q, u)$ and there is a
costate state $p$ such that the triplet $(q,p,u)$ fulfills the first-order optimality condition as follow:
\begin{align}\label{eq03}
&\left\{
             \begin{array}{lll}\displaystyle
  \kappa_1\frac{\partial^{\alpha} u(x,y)}{\partial|x|^{\alpha}}
  +\kappa_2\frac{\partial^{\alpha} u(x,y)}{\partial|y|^{\alpha}}=g(x,y)+q(x,y), \quad (x,y) \in\Omega, \\
       u(x,y)=0, \quad (x,y)\in\partial\Omega,
             \end{array}
        \right.\\
\label{eq04}
&\left\{
             \begin{array}{lll}\displaystyle
  \kappa_1\frac{\partial^{\alpha} p(x,y)}{\partial|x|^{\alpha}}
  +\kappa_2\frac{\partial^{\alpha} p(x,y)}{\partial|y|^{\alpha}}=u(x,y)-u_d(x,y), \quad (x,y) \in\Omega, \\
       p(x,y)=0, \quad (x,y)\in\partial\Omega,
             \end{array}
        \right. \\
&\int_\Omega(\gamma q+p)(\delta q-q)dxdy\geq 0,\quad \forall\delta q\in \mathscr{K}.\label{eq05}
\end{align}
\end{theorem}
\begin{proof}
Due to the strictly convex $J(\cdot,\cdot)$, we easily know that the OCPs (\ref{eq031}) admit a unique pair
$(q, u)$ by standard arguments. Next, we prove the first-order optimality condition (\ref{eq03})-(\ref{eq05}).
Suppose that $v(x,y)$ is the state with respect to $\delta q(x,y)-q(x,y)$, i.e.,
\begin{align}\label{eq06}
&\left\{
             \begin{array}{lll}\displaystyle
  \kappa_1\frac{\partial^{\alpha} v(x,y)}{\partial|x|^{\alpha}}
  +\kappa_2\frac{\partial^{\alpha} v(x,y)}{\partial|y|^{\alpha}}=\delta q(x,y)-q(x,y), \quad (x,y) \in\Omega, \\
       v(x,y)=0, \quad (x,y)\in\partial\Omega.
             \end{array}
        \right.
\end{align}

Define the reduced cost functional $\hat{J}(q):=J(q,u(q))$, which maps $q$ from
$\mathscr{K}$ to $\mathbb{R}$. Then the first-order optimality condition reads as
\begin{equation*}
  \hat{J}'(q)(\delta q-q)\geq0, \quad \forall \delta q\in \mathscr{K}.
\end{equation*}
which leads to
\begin{equation}\label{eq08}
  \int_\Omega\gamma q(\delta q-q)dxdy+\int_\Omega v(u-u_d)dxdy\geq 0,\quad \forall\delta q\in \mathscr{K}.
\end{equation}

On the other hand, we present the adjoint state equation
\begin{align}\label{eq09}
&\left\{
             \begin{array}{lll}\displaystyle
  \kappa_1\frac{\partial^{\alpha} p(x,y)}{\partial|x|^{\alpha}}
  +\kappa_2\frac{\partial^{\alpha} p(x,y)}{\partial|y|^{\alpha}}=u(x,y)-u_d(x,y), \quad (x,y) \in\Omega, \\
       p(x,y)=0, \quad (x,y)\in\partial\Omega,
             \end{array}
        \right.
\end{align}
with the costate $p$. Multiplying by $v$ and using Lemma \ref{lemma01}, there holds
\begin{small}
\begin{align*}
  \int_\Omega v(u-u_d)dxdy&=\kappa_1\int_\Omega v\cdot\frac{\partial^{\alpha} p(x,y)}{\partial|x|^{\alpha}}dxdy
    +\kappa_2\int_\Omega v\cdot\frac{\partial^{\alpha} p(x,y)}{\partial|y|^{\alpha}}dxdy\\
    &=\frac{-\kappa_1}{2\cos(\frac{\pi\alpha}{2})}\int_\Omega
        {_{a}^{L}D^{\alpha}_x} p\cdot v+^{R}_{x}D^{\alpha}_b p\cdot vdxdy
     -\frac{\kappa_2}{2\cos(\frac{\pi\alpha}{2})}\int_\Omega{_{c}^{L}D^{\alpha}_y} p\cdot v+^{R}_{y}D^{\beta}_d p\cdot vdxdy\\
    &=\frac{-\kappa_1}{2\cos(\frac{\pi\alpha}{2})}\int_\Omega
        p\cdot ^{R}_{x}D^{\alpha}_bv+ p\cdot {_{a}^{L}D^{\alpha}_x} vdxdy
     -\frac{\kappa_2}{2\cos(\frac{\pi\alpha}{2})}\int_\Omega p\cdot ^{R}_{y}D^{\alpha}_dv+ p\cdot {_{c}^{L}D^{\alpha}_y} vdxdy\\
    &=\kappa_1\int_\Omega p\cdot\frac{\partial^{\alpha} v(x,y)}{\partial|x|^{\alpha}}dxdy
    +\kappa_2\int_\Omega p\cdot\frac{\partial^{\alpha} v(x,y)}{\partial|y|^{\alpha}}dxdy.
\end{align*}
\end{small}
Combing with Eq. (\ref{eq06}), we obtain
\begin{equation}\label{eq11}
   \int_\Omega v(u-u_d)dxdy=\int_\Omega p(\delta q-q)dxdy,
\end{equation}
and substituting Eq. (\ref{eq11}) into (\ref{eq08}) finally leads to the above results.
\end{proof}

\section{Fully discrete Galerkin FE scheme} \label{s3}

In order to derive the FE scheme, divide $\Omega$ by triangle meshes $\mathcal{T}_h$ and for each
triangle $K$, let $h_K=\textrm{diam}\ K$ and $h=\max\limits_{K\in\mathcal{T}_h}h_K$.
Define the FE subspace $\mathcal{V}_h=\{v:v|_{K}\in P_{linear}, \ \forall K\in\mathcal{T}_h\}$ and
$\mathcal{V}_h\in H^{\frac{\alpha}{2}}_0(\Omega)$, where $P_{linear}$ is the linear polynomial space.
Using fractional variational principle, the FE scheme for state Eq. (\ref{eq02}) is to find $u_h\in\mathcal{V}_h$ such that
\begin{align}\label{eq12}
   \varLambda_h(u_h,\chi_h)=(g+q,\chi_h),\quad \forall \chi_h\in\mathcal{V}_h,
\end{align}
where
\begin{align*}
    \varLambda_h(u,v)&= \frac{\kappa_1}{2\cos(\frac{\pi\alpha}{2})}
     \Bigg\{({^L_a}D^{\frac{\alpha}{2}}_x u,{_x}^RD^{\frac{\alpha}{2}}_b v)+({_x^R}D^{\frac{\alpha}{2}}_b u,{_a^L}D^{\frac{\alpha}{2}}_x v)\Bigg\}\\
    & \ \ + \frac{\kappa_2}{2\cos(\frac{\pi\alpha}{2})}\Bigg\{({_c^L}D^{\frac{\alpha}{2}}_y u,{_y^R}D^{\frac{\alpha}{2}}_d v)
          +({_y^R}D^{\frac{\alpha}{2}}_d u,{_c^L}D^{\frac{\alpha}{2}}_y v)\Bigg\},
\end{align*}
which satisfies $\varLambda_h(u,v) \leq C||u||_{eng}||v||_{eng}$, $\varLambda_h(u,u) \geq C||u||^2_{eng}$
with the energy norms
\begin{equation*}
||u||_{eng}=\big(||u||^2_{L^2(\Omega)}+|u|^2_{eng}\big)^{\frac{1}{2}},\quad
|u|_{eng}=|\varLambda_h(u,u)|^{\frac{1}{2}}.
\end{equation*}
which is  equivalent to $||u||_{H^\frac{\alpha}{2}}(\Omega)$ \cite{ref21}.
Denote the $L^2$ projection of $u$ by $\mathcal{R}_hu$ and the piecewise polynomial interpolant of $u$ by
$\Pi_hu$, which have the below properties \cite{ref23}:
\begin{align}
    ||u-\mathcal{R}_h u||_{L^2(\Omega)}&\leq Ch^{r}||u||_{H^{r}(\Omega)},\label{eq22}\\
    ||u-{\Pi}_h u||_{H^s(\Omega)}&\leq Ch^{r-s}||u||_{H^{r}(\Omega)},\quad 0\leq s\leq r.\label{eq13}
\end{align}

In addition, we define the elliptic projection $\mathcal{P}_h:H_0^\frac{\alpha}{2}(\Omega)\mapsto\mathcal{V}_h$ by
\begin{equation*}
\varLambda_h(u,\chi_h)=\varLambda_h(\mathcal{P}_h u, \chi_h),\quad \forall \chi_h\in \mathcal{V}_h,
\end{equation*}
which satisfies the following approximate property.
\begin{lemma}\label{le06}
\cite{ref24} Let $u\in H^{r}(\Omega)\cap\mathcal{V}_h$. Then we have
\begin{align}\label{eq40}
    ||u-\mathcal{P}_h u||_{H^\frac{\alpha}{2}(\Omega)}\leq C h^{r-\frac{\alpha}{2}}||u||_{H^r(\Omega)},\ \ \alpha<2r,
\end{align}
with a constant $C$ independent of $h$.
\end{lemma}

We can derive the following convergent result for the above FE scheme.
\begin{lemma}\label{le01}
Let $q=0$ and $u\in H_0^{1+\frac{\alpha}{2}}(\Omega)$. Then there exists a constant $C$ unrelated to $h$ such that
\begin{equation}\label{eq14}
  ||u-u_h||_{H^{\frac{\alpha}{2}}(\Omega)}\leq Ch||u||_{H^{1+\frac{\alpha}{2}}(\Omega)}.
\end{equation}
\end{lemma}
\begin{proof}
Using Galerkin orthogonality, we have
\begin{align*}
  C||u-u_h||_{eng}&\leq \varLambda_h(u-u_h,u-u_h)=\varLambda_h(u-u_h,u-\chi_h)\\
                   &\leq \bar{C}||u-u_h||_{eng}||u-\chi_h||_{eng}.
\end{align*}
with $\bar{C}$ independent of $h$. Since the equivalence of $||\cdot||_{eng}$ and
$||\cdot||_{H^{\frac{\alpha}{2}}(\Omega)}$, it implies that
\begin{align*}
  ||u-u_h||_{H^{\frac{\alpha}{2}}(\Omega)}\leq C\inf\limits_{\chi_h\in\mathcal{V}_h}||u-\chi_h||_{H^{\frac{\alpha}{2}}(\Omega)}.
\end{align*}
Taking $\chi_h=\Pi_hu$ and noticing (\ref{eq13}), we finally obtain
\begin{align*}
  ||u-u_h||_{H^{\frac{\alpha}{2}}(\Omega)}\leq C||u-\Pi_hu||_{H^{\frac{\alpha}{2}}(\Omega)}
                                           \leq Ch||u||_{H^{1+\frac{\alpha}{2}}(\Omega)},
\end{align*}
which ends the proof.
\end{proof}

Letting $q\neq0$, the FE scheme for Eqs. (\ref{eq01})-(\ref{eq02}) is to find a pair $(q_h,u_h)$ such that
\begin{align}\label{eq16}
  \textrm{Minimize}\ J(u_h,q_h) \ \textrm{subjected to Eq.}\ (\ref{eq12}),
                   \quad (q_h,u_h)\in \mathscr{K}\times\mathcal{V}_h,
\end{align}
which is equivalent to find the triplet $(q_h, p_h,u_h)$ fulfilling the discrete optimality condition:
\begin{align}\label{eq17}
&\left\{
             \begin{array}{lll}\displaystyle
       \varLambda_h(u_h,\chi_h)=(g+q_h,\chi_h),\quad \forall \chi_h\in\mathcal{V}_h, \\
       u_h(x,y)=0, \quad (x,y)\in\partial\Omega,
             \end{array}
        \right.\\
\label{eq18}
&\left\{
             \begin{array}{lll}\displaystyle
       \varLambda_h(p_h,\chi_h)=(u_h-u_d,\chi_h),\quad \forall \chi_h\in\mathcal{V}_h, \\
       p_h(x,y)=0, \quad (x,y)\in\partial\Omega,
             \end{array}
        \right. \\
&\int_\Omega(\gamma q_h+p_h)(\delta q_h-q_h)dxdy\geq 0,\quad \forall\delta q_h\in \mathscr{K}.\label{eq19}
\end{align}

Due to the variational inequality, the control variable always has low regularity.
To overcome this drawback, we use the variational discretization method to treat $q$,
i.e., (\ref{eq19}) is recast as
\begin{equation}\label{eq20}
  q_h=P_{\mathscr{K}}\Bigg(-\frac{1}{\gamma}p_h\Bigg)=\max\Bigg\{v_1,\min\Bigg(-\frac{1}{\gamma}p_h,v_2\Bigg)\Bigg\},
\end{equation}
where $P_{\mathscr{K}}$ is termed by pointwise projection operator.

\section{Error estimates} \label{s4}
In this section, we establish the convergent analysis for the above FE scheme (\ref{eq17})-(\ref{eq20})
and to this end, we introduce the auxiliary variational equations:
\begin{align}
    \varLambda_h(u_h(q),\chi_h)&=(g+q,\chi_h),\quad \forall \chi_h\in\mathcal{V}_h, \label{eq24} \\
    \varLambda_h(p_h(q),\chi_h)&=(u_h(q)-u_d,\chi_h),\quad \forall \chi_h\in\mathcal{V}_h. \label{eq25}
\end{align}
Obviously, $u_h(q)$ is the FE solution of state $u$ and by Lemma \ref{le01}, there exists
\begin{align}
  ||u-u_h(q)||_{H^\frac{\alpha}{2}(\Omega)}\leq Ch||u||_{H^{1+\frac{\alpha}{2}}(\Omega)}. \label{eq27}
\end{align}

\begin{lemma}\label{le02}
If $(q,p,u)$ are the analytical solutions of the OCPs (\ref{eq031}), $(q_h, p_h,u_h)$
are the FE solutions obtained by (\ref{eq17})-(\ref{eq20}) and $q\in H^1(\Omega)$, then we have
\begin{equation}\label{eq26}
  ||q-q_h||_{L^2(\Omega)}\leq Ch+C||p-p_h(q)||_{L^2(\Omega)},
\end{equation}
where $C$ is  a constant unrelated to $h$.
\end{lemma}
\begin{proof}
Using  Eqs. (\ref{eq17})-(\ref{eq18}) and Eqs. (\ref{eq24})-(\ref{eq25}), we find
\begin{align*}
   (q-q_h,\chi_h)&=\Lambda_h(u_h(q)-u_h,\chi_h),\\
   (u_h(q)-u_h,\chi^*_h)&=\Lambda_h(p_h(q)-p_h,\chi^*_h),\quad \forall \chi_h,\ \chi^*_h\in \mathcal{V}_h,
\end{align*}
and letting $\chi_h=p_h(q)-p_h$, $\chi^*_h=u_h(q)-u_h$ leads to
\begin{align*}
   (q-q_h,p_h(q)-p_h)=(u_h(q)-u_h,u_h(q)-u_h)\geq0.
\end{align*}

From the above inequality, it follows that
\begin{align*}
  \gamma||q-q_h||^2_{L^2(\Omega)}&= (\gamma q+p_h(q),q-q_h)-(\gamma q_h+p_h,q-q_h)-(p_h(q)-p_h,q-q_h)\\
         &\leq (\gamma q+p_h(q),q-q_h)-(\gamma q_h+p_h,q-q_h)\\
         &\leq (\gamma q+p,q-q_h)+(p_h(q)-p,q-q_h) \\
         &\quad -(\gamma q_h+p_h,q-\mathcal{R}_hq)+(\gamma q_h+p_h,q_h-\mathcal{R}_hq).
\end{align*}
Meanwhile, by virtue of (\ref{eq05}) and (\ref{eq19}), we have
\begin{align*}
  (\gamma q+p,q-q_h)\leq 0,\quad
  (\gamma q_h+p_h,q_h-\mathcal{R}_hq)\leq 0,
\end{align*}
and then it suffices to prove that
\begin{align}\label{eq310}
\begin{aligned}
  \gamma||q-q_h||^2_{L^2(\Omega)}&\leq ( p_h(q)-p,q-q_h)-(\gamma q_h+p_h,q-\mathcal{R}_hq)\\
         &=( p_h(q)-p,q-q_h)+\gamma( q-q_h,q-\mathcal{R}_hq)+(p-p_h(q),q-\mathcal{R}_hq) \\
         &\quad+(p_h(q)-p_h,q-\mathcal{R}_hq)-(\gamma q+p,q-\mathcal{R}_hq)\\
         &=( p_h(q)-p,q-q_h)+\gamma( q-q_h,q-\mathcal{R}_hq)\\
         &\quad +(p-p_h(q),q-\mathcal{R}_hq)-(\gamma q+p,q-\mathcal{R}_hq).
\end{aligned}
\end{align}

Furthermore, by using the properties of $\mathcal{R}_h$ and $q\in H^1(\Omega)$, there exists
\begin{align}\label{eq30}
\begin{aligned}
  (\gamma q+p,q-\mathcal{R}_hq)&=(\gamma q+p -\mathcal{R}_h(\gamma q+p),q-\mathcal{R}_hq)\\
             &\leq||\gamma q+p -\mathcal{R}_h(\gamma q+p)||_{L^2(\Omega)}||q-\mathcal{R}_hq||_{L^2(\Omega)}\leq Ch^2.
\end{aligned}
\end{align}
Applying (\ref{eq22}), (\ref{eq30}) and Young's inequality to (\ref{eq310}), we have
\begin{align*}
  \gamma||q-q_h||^2_{L^2(\Omega)}&\leq \varepsilon||q-q_h||^2_{L^2(\Omega)} +C_\varepsilon ||p-p_h(q)||^2_{L^2(\Omega)} \\
         &\quad+C_\varepsilon ||q-\mathcal{R}_hq||^2_{L^2(\Omega)}  - (\gamma q+p,q-\mathcal{R}_hq)\\
         &\leq \varepsilon||q-q_h||^2_{L^2(\Omega)} +Ch^2+C_\varepsilon||p-p_h(q)||^2_{L^2(\Omega)}.
\end{align*}
By taking $\varepsilon<\gamma$, we finally obtain
\begin{align*}
  ||q-q_h||_{L^2(\Omega)}\leq Ch+C||p-p_h(q)||_{L^2(\Omega)},
\end{align*}
and this completes the proof.
\end{proof}

To derive the error bounds, we further give the auxiliary equation
\begin{align}
    \varLambda_h(p_h(u),\chi_h)&=(u-u_d,\chi_h),\quad \forall \chi_h\in\mathcal{V}_h, \label{eq31}
\end{align}
and obviously, $p_h(u)$ is the FE solution of costate $p$, which satisfies
\begin{align}
  ||p-p_h(u)||_{H^\frac{\alpha}{2}(\Omega)}\leq Ch||p||_{H^{1+\frac{\alpha}{2}}(\Omega)}. \label{eq32}
\end{align}

Based on the above discussions, we have the following priori error estimates.
\begin{theorem}
If $(q,p,u)$ are the analytical solutions of the OCPs (\ref{eq031}), $(q_h, p_h,u_h)$
are the FE solutions obtained by (\ref{eq17})-(\ref{eq20}) and $q\in H^1(\Omega)$,
$p, u\in H_0^{1+\frac{\alpha}{2}}(\Omega)$, then we have
\begin{equation}\label{eq26}
  ||q-q_h||_{L^2(\Omega)}+||p-p_h||_{H^\frac{\alpha}{2}(\Omega)}+||u-u_h||_{H^\frac{\alpha}{2}(\Omega)}\leq Ch,
\end{equation}
where $C$ is a constant unrelated to $h$.
\end{theorem}
\begin{proof}
Multiplying Eq. (\ref{eq04}) by $\chi_h\in \mathcal{V}_h$ and subtracting Eq. (\ref{eq25}), we have
\begin{align*}
  C||p-p_h(q)||^2_{eng}&\leq \varLambda_h(p-p_h(q),p-p_h(q))\\
                   &\leq \varLambda_h(p-p_h(q),\mathcal{P}_hp-p_h(q))+\varLambda_h(p-p_h(q),p-\mathcal{P}_hp)\\
                   &= (u-u_h(q),\mathcal{P}_hp-p_h(q))\\
                   &= (u-u_h(q),p-p_h(q))+(u-u_h(q),\mathcal{P}_hp-p),
\end{align*}
by taking $\chi_h=p-p_h(q)$ in both two equations. 

Using the equivalence of $||\cdot||_{eng}$ and $||\cdot||_{H^{\frac{\alpha}{2}}(\Omega)}$ and Lemma \ref{le06}, there exists
\begin{align*}
    ||p-p_h(q)||^2_{H^\frac{\alpha}{2}(\Omega)}&\leq \delta||p-p_h(q)||^2_{L^2(\Omega)}+C_\delta||u-u_h(q)||^2_{L^2(\Omega)}
              + C_\delta||p-\mathcal{P}_hp||^2_{H^{\frac{\alpha}{2}}}\\
              &\leq \delta||p-p_h(q)||^2_{L^2(\Omega)}+C_\delta||u-u_h(q)||^2_{L^2(\Omega)}+ Ch^2||p||^2_{H^{1+\frac{\alpha}{2}}(\Omega)},
\end{align*}
with $0<\delta\ll 1$, which implies
\begin{align*}
    ||p-p_h(q)||_{H^\frac{\alpha}{2}(\Omega)}&\leq Ch||p||_{H^{1+\frac{\alpha}{2}}(\Omega)}+C||u-u_h(q)||_{H^{\frac{\alpha}{2}}}.
\end{align*}
Combining with (\ref{eq27}) and Lemma \ref{le02}, we obtain
\begin{align}
\begin{aligned}\label{eq35}
  ||q-q_h||_{L^2(\Omega)}\leq Ch+C||p-p_h(q)||_{H^\frac{\alpha}{2}(\Omega)}
                         \leq Ch+C||u-u_h(q)||_{L^2(\Omega)}\leq Ch.
\end{aligned}
\end{align}

Subtracting Eq. (\ref{eq17}) from Eq. (\ref{eq24}) and taking $\chi_h=u_h(q)-u_h$, it holds that
\begin{align*}
  C||u_h(q)-u_h||^2_{eng}&\leq \varLambda_h(u_h(q)-u_h,u_h(q)-u_h)=(q-q_h,u_h(q)-u_h)\\
                   &\leq ||q-q_h||_{L^2(\Omega)}||u_h(q)-u_h||_{eng}.
\end{align*}
Then, based on the error bound of $q-q_h$, we can get
\begin{align}\label{eq36}
  ||u_h(q)-u_h||_{H^{\frac{\alpha}{2}}(\Omega)}&\leq C ||q-q_h||_{L^2(\Omega)}\leq Ch,
\end{align}
and similarly, from the difference of Eqs. (\ref{eq18}) and (\ref{eq31}), it follows that
\begin{align}\label{eq37}
  ||p_h(u)-p_h||_{H^{\frac{\alpha}{2}}(\Omega)}&\leq C ||u-u_h||_{L^2(\Omega)}.
\end{align}
Using (\ref{eq27}), (\ref{eq32}), (\ref{eq36}), (\ref{eq37}) and triangle inequality, we obtain
\begin{align}
\begin{aligned}\label{eq39}
  ||u-u_h||_{H^{\frac{\alpha}{2}}(\Omega)}&\leq||u-u_h(q)||_{H^{\frac{\alpha}{2}}(\Omega)}+||u_h(q)-u_h||_{H^{\frac{\alpha}{2}}(\Omega)}\\
                         &\leq Ch||u||_{H^{1+\frac{\alpha}{2}}(\Omega)}+C||q-q_h||_{L^2(\Omega)}\leq Ch,
\end{aligned}\\
\begin{aligned}\label{eq38}
 ||p-p_h||_{H^{\frac{\alpha}{2}}(\Omega)}&\leq ||p-p_h(u)||_{H^{\frac{\alpha}{2}}(\Omega)}+ ||p_h(u)-p_h||_{H^{\frac{\alpha}{2}}(\Omega)}\\
                         &\leq Ch||p||_{H^{1+\frac{\alpha}{2}}(\Omega)}+ C ||u-u_h||_{H^{\frac{\alpha}{2}}}\leq Ch.
\end{aligned}
\end{align}

Consequently, combining (\ref{eq35}), (\ref{eq39}) and (\ref{eq38}), we obtain the priori error estimate.
\end{proof}

\section{Illustrative test} \label{s5}

In this section, to illustrate the accuracy performance of the proposed FE scheme,
numerical tests are carried out and numerical results are presented.
For solving the coupled system (\ref{eq17})-(\ref{eq20}), we adopt the fixed-point iterative algorithm
and terminate iterative loop by reaching a solution $q_h$ with tolerant error 1.0e-12.
We employ piecewise linear interpolation to approximate $p$, $u$ and
variational discretization method to discretize $q$.
Meanwhile, denote
\begin{align*}
Cov.\ order=
\log_{2}\Bigg(\frac{e_{h_1}}{e_{h_2}}\Bigg)
      \Bigg/\log_{2}\Bigg(\frac{h_1}{h_2}\Bigg),
\end{align*}
where $e_{h_k}$ is the global error corresponding to the meshsize $h_k$, $k=1,\ 2$.\\

\noindent
\textbf{Example 1.}
Letting $\kappa_1=\kappa_2=1$, $\gamma=1$ and $\mathscr{K}=\{q\in L^2(\Omega):\ -3\leq q(x,y)\leq -0.1 \}$,
consider the problem on $\Omega=(0,1)\times(0,1)$ with the analytic solutions
\begin{align*}
  u&=10x(1-x)y(1-y),\\
  p&=5x(1-x)y(1-y), \\
  q&=\max\big\{-3,\min\{-p,-0.1\}\big\},
\end{align*}
where $g$, $u_d$ are determined by $u$, $p$ and $q$ via the model of OCPs.

We evaluate the global error at coarse mesh and then refine the
mesh several times. In Fig. \ref{fig1}, we show the decline behavior of error for the control,
state and adjoint state variables with different $\alpha$ in log-log scale. In Fig. \ref{fig2}, we
present an unstructured mesh of $h=1/25$ and compare the analytic solution with the  approximation
of state when $\alpha=1.9$. To obtain more insight about
accuracy, letting $\alpha=1.3$, we compute the error with different
$h$ and report the convergent order for control, state and adjoint state variables in
Table \ref{tab1}. From the above results, we observe that our method is almost convergent
with theoretical order and yields the solution indistinguishable from the analytic solution, which
confirm the convergent accuracy and theoretical analysis.  \\

\begin{figure*}[!htb]
\centering
\begin{minipage}[t]{0.32\linewidth}
\includegraphics[width=1.8in]{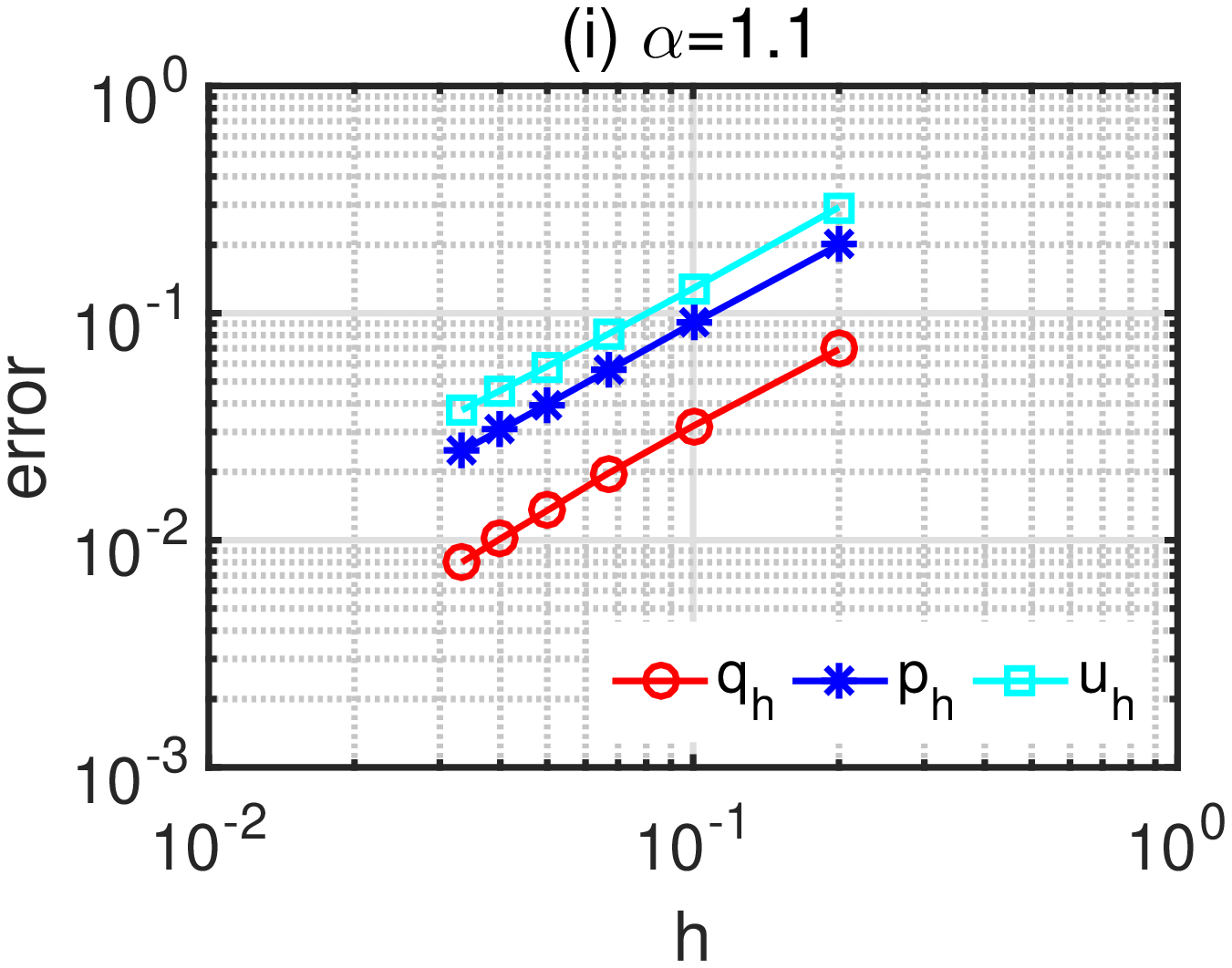}
\end{minipage}
\begin{minipage}[t]{0.32\linewidth}
\includegraphics[width=1.8in]{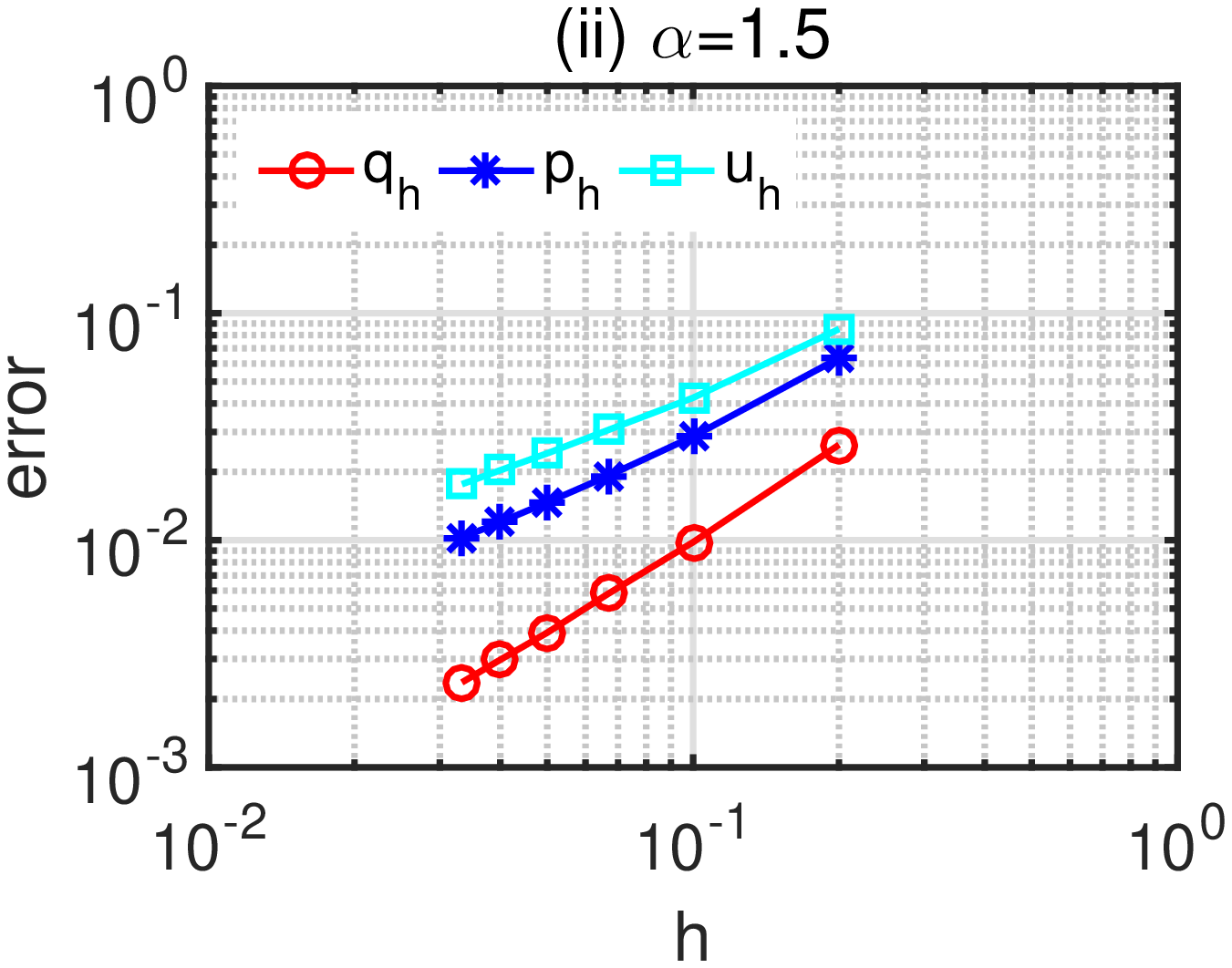}
\end{minipage}
\begin{minipage}[t]{0.32\linewidth}
\includegraphics[width=1.8in]{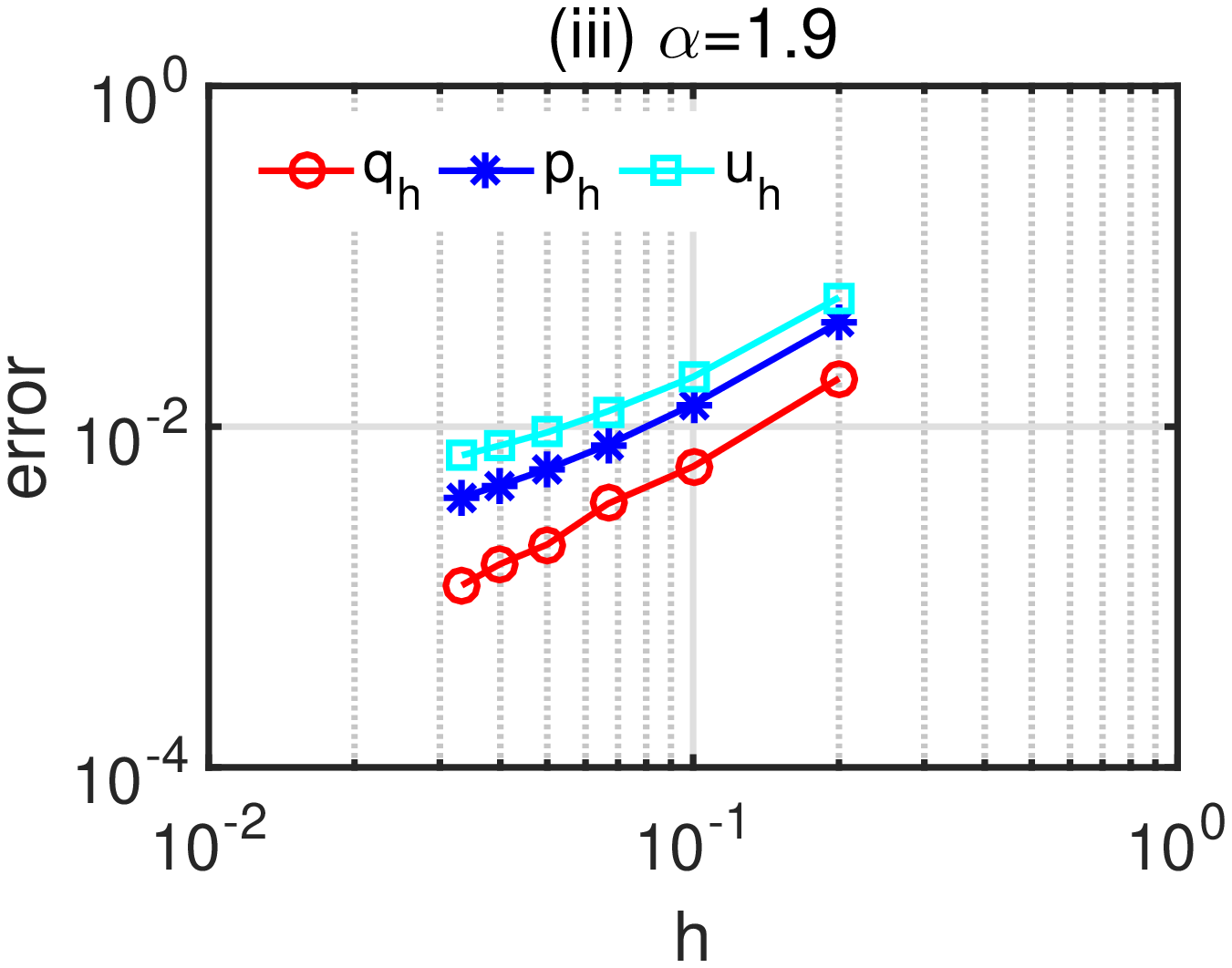}
\end{minipage}
\caption{Evolution of error versus the change of $h$ for different $\alpha$:
                  (i) $\alpha=1.1$; (ii) $\alpha=1.5$; (iii) $\alpha=1.9$.}\label{fig1}
\end{figure*}

\begin{figure*}[!htb]
\centering
\begin{minipage}[t]{0.32\linewidth}
\includegraphics[width=1.9in]{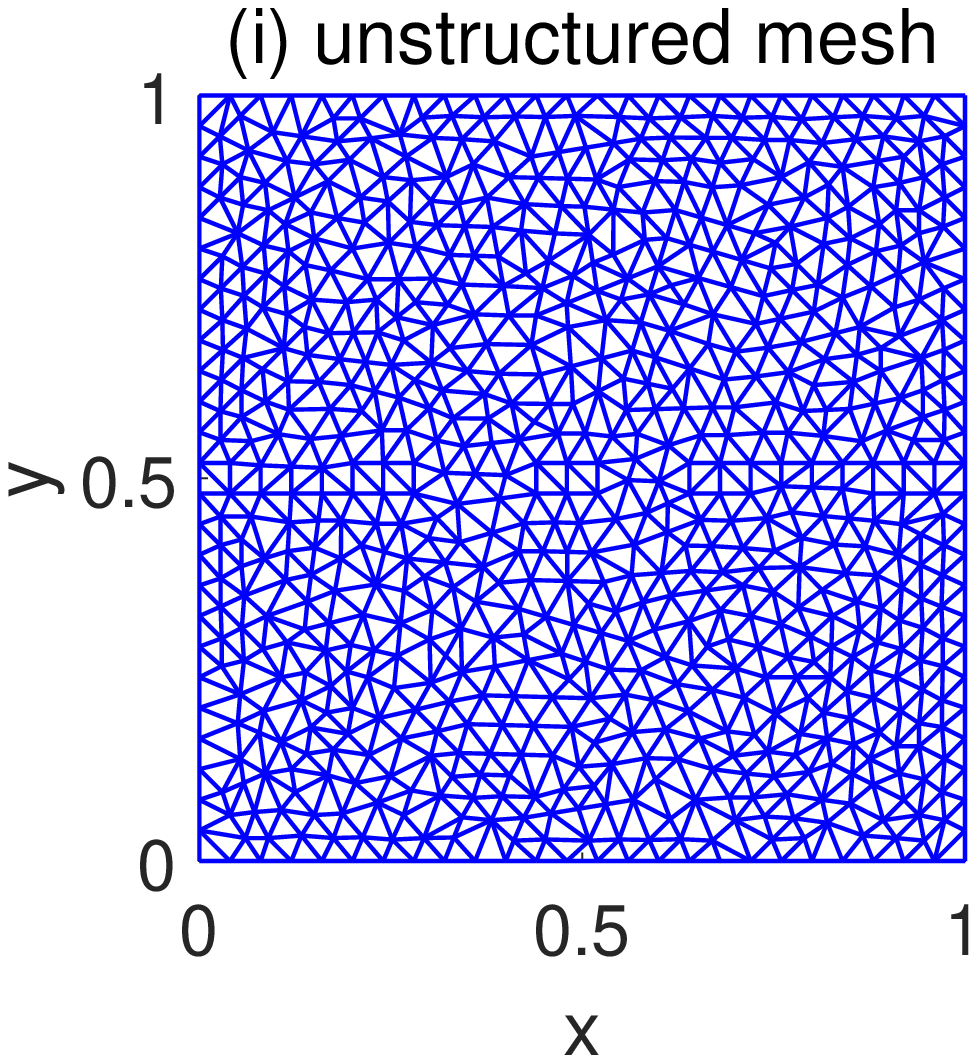}
\end{minipage}
\begin{minipage}[t]{0.33\linewidth}
\includegraphics[width=1.9in]{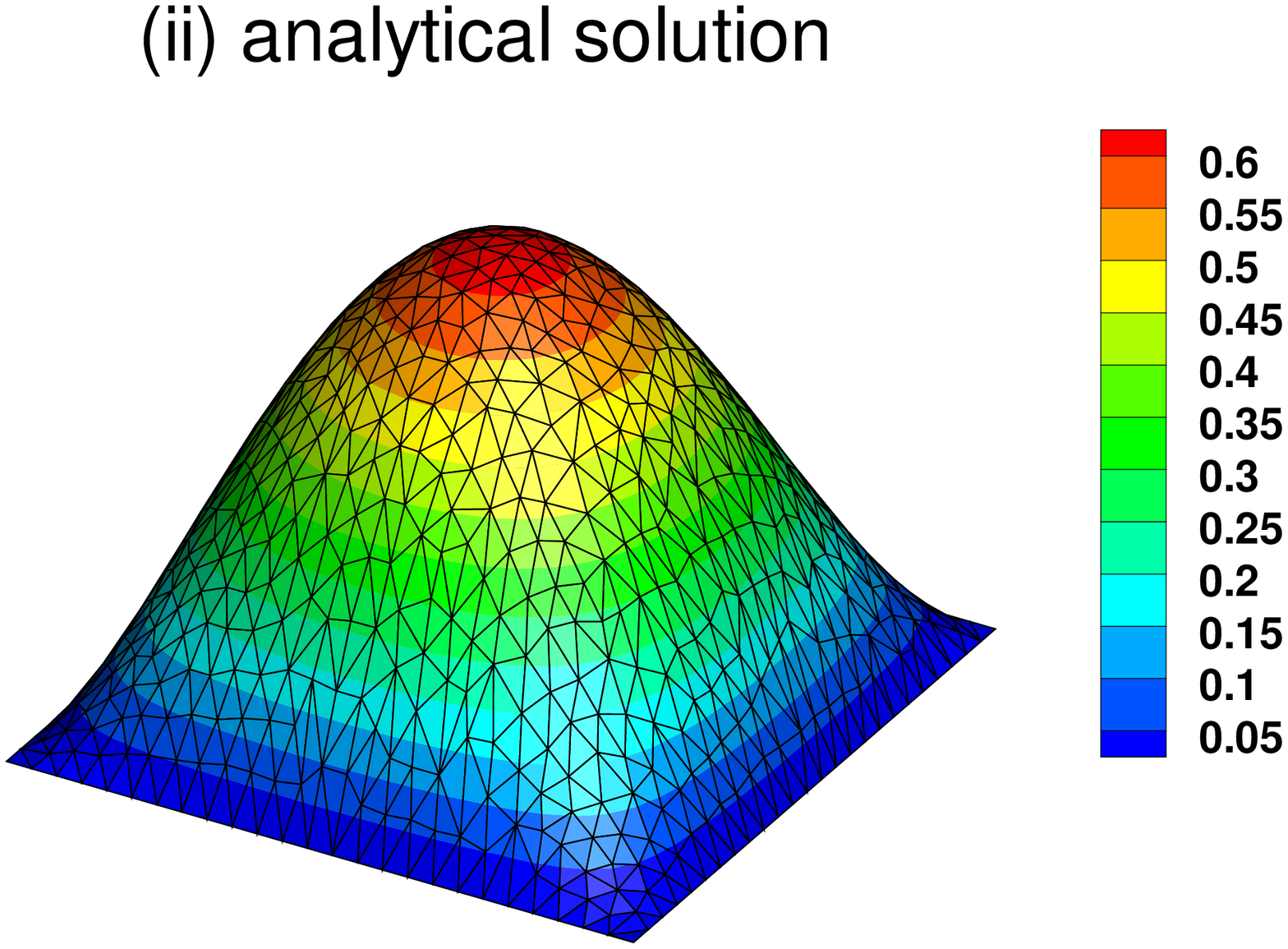}
\end{minipage}
\begin{minipage}[t]{0.32\linewidth}
\includegraphics[width=1.9in]{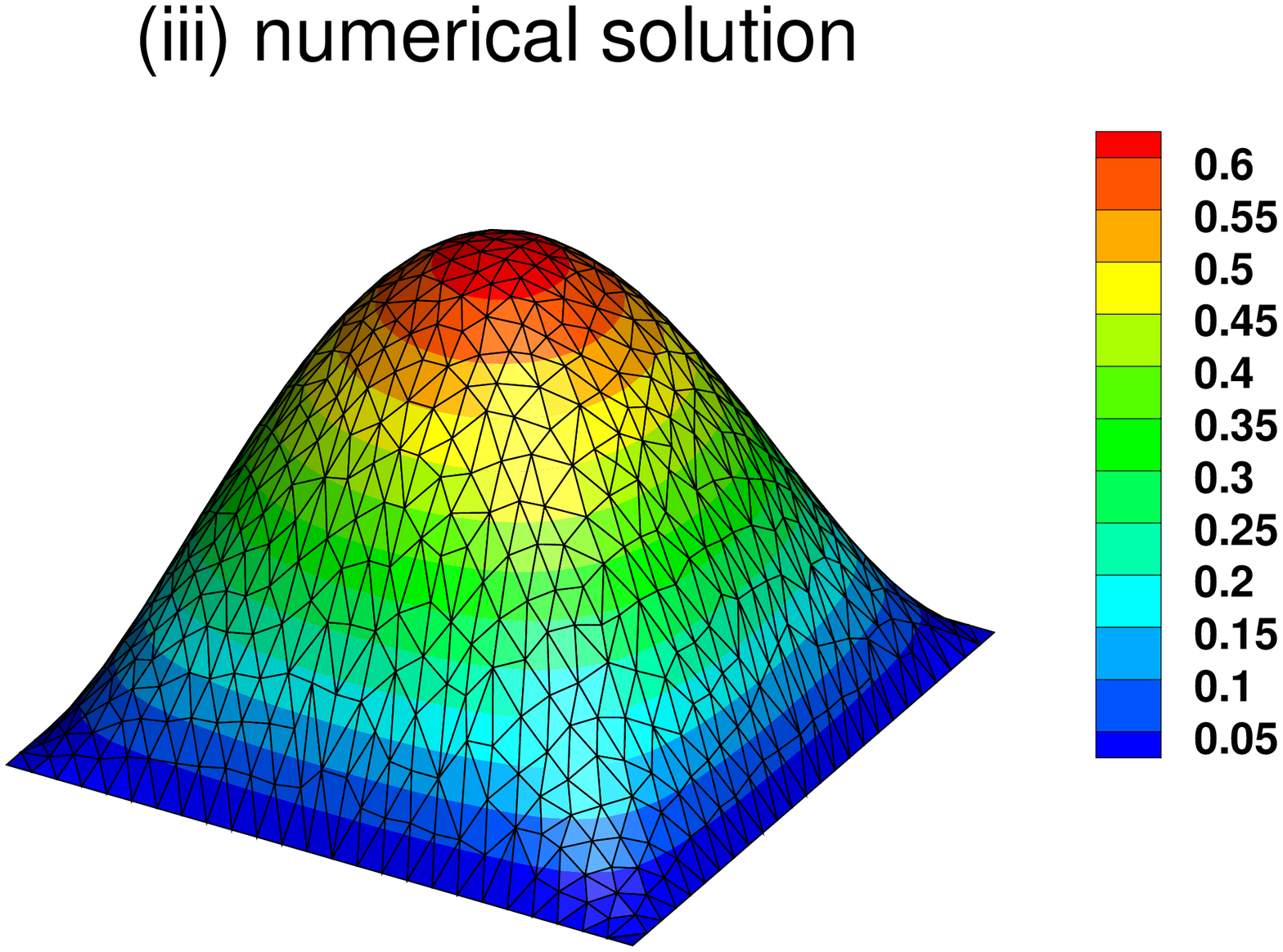}
\end{minipage}
\caption{Unstructured triangular mesh, analytical and numerical solutions of state: (i) $h=1/25$;
             (ii) $u$; (iii) $u_h$.}\label{fig2}
\end{figure*}

\begin{table*}[!htb]
\centering
\caption{The global error and convergent order for $q_h$, $p_h$ and $u_h$ when $\alpha=1.3$. } \label{tab1}
\begin{tabular}{ccccccc}
\toprule      $h$  & $||q-q_h||_{L^2(\Omega)}$  & $Cov.\ order$  & $||p-p_h||_{H^\frac{\alpha}{2}(\Omega)}$
                 & $Cov.\ order$  & $||u-u_h||_{H^\frac{\alpha}{2}(\Omega)}$  & $Cov.\ order$\\
\hline  1/10   &1.4265e-02  &-      &4.0697e-02  &-     &5.8055e-02  &-    \\
        1/15   &8.4741e-03  &1.28   &2.6194e-02  &1.09  &3.9514e-02  &0.95  \\
        1/20   &5.8169e-03  &1.31   &1.9352e-02  &1.05  &3.0027e-02  &0.96   \\
\bottomrule
\end{tabular}
\end{table*}

\noindent
\textbf{Acknowledgement}:
This research was supported by the Natural Science Foundation of Hunan Province of China (No. 2020JJ5514)
and the Scientific Research Funds of Hunan Provincial Education Department (Nos. 19C1643, 19B509 and 20B532).


\end{document}